\documentclass[journal,twoside,web]{ieeecolor}
\usepackage{generic}
\usepackage{cite}
\usepackage{amsmath,amssymb,amsfonts}
\usepackage{mathtools}
\usepackage{graphicx}
\usepackage[hidelinks]{hyperref}
\usepackage{textcomp}
\def\BibTeX{{\rm B\kern-.05em{\sc i\kern-.025em b}\kern-.08em
    T\kern-.1667em\lower.7ex\hbox{E}\kern-.125emX}}

\newtheorem{thm}{Theorem}

\newtheorem{prop}{Proposition}
\newtheorem{cor}{Corollary}
\newtheorem{asmp}{Assumption}

\DeclareMathOperator{\sgn}{sgn}

\usepackage{xcolor}
\newcommand{\addedcolor}{black}
\newcommand{\movedcolor}{black}
\newcommand{\add}[1]{\textcolor{\addedcolor}{#1}}

\allowdisplaybreaks

\begin{document}

\title{Lyapunov Constructions for System Interconnections Arising from Adaptation\\ in Some Optimization Methods}

\author{Qixu Wang$^{1,\ast}$, Patrick McNamee$^{1,\ast}$, Zahra Nili Ahmadabadi$^{1}$, and Miroslav Krsti{\'c}$^{2}$
\thanks{Research was sponsored by the Army Research Office under Grant
Number W911NF-24-1-0386 and the Department of the Navy, Office of Naval Research
under ONR award number N000142412269. The views, findings, conclusions, or recommendations contained in this document are those of the authors and should not be interpreted as representing the official policies or views, either expressed or implied, of the Army Research Office, the Office of Naval Research, or the U.S. Government. The U.S. Government is authorized to reproduce and distribute reprints for Government purposes notwithstanding any copyright notation herein.}
\thanks{$^{1}$Q. Wang, P. McNamee, and Z. N. Ahmadabadi are with the Department of Mechanical Engineering, San Diego State University, San Diego, CA, USA
        {\tt\small $\lbrace$qwang0429, pmcnamee5123,zniliahmadabadi$\rbrace$@sdsu.edu}}%
\thanks{$^{2}$M. Krsti{\'c} is with the Department of Mechanical and Aerospace Engineering, University of California San Diego, La Jolla, CA, USA
        {\tt\small krstic@ucsd.edu}}%
\thanks{$^{\ast}$ These authors contributed equally as joint first authors.}
}

\maketitle

\begin{abstract}
Interconnected systems have been widely studied, with a focus on interconnected systems whose subsystems are solely input-to-state stable (ISS) or passive systems. The focus of this work is on the interconnected systems that appear in adaptive gradient methods. In adaptive gradient methods, one subsystem seeks to move parameters of a cost function towards a minimizer while the other subsystem works to estimate some derivative information about the cost function to help determine the direction of the parameter update. This work studies instances of such interconnected systems and gives various Lyapunov function \add{constructions for them} using different techniques. \add{Global asymptotic stability (GAS) of adaptive gradient optimizers follows from small-gain theory; the Lyapunov functions constructed here certify it explicitly.}
\end{abstract}

\section{Introduction}
    \label{sec:introduction}

\textcolor{\movedcolor}{Adaptive gradient methods are a prominent class of algorithms in optimization, particularly in neural network training \cite{ref:ma-2021,ref:yao-2021}. These methods use adaptive gains to alter the standard gradient descent optimizer, with the adaptation laws being based on the time history of either first- or second-order derivative information. These adaptation laws allow for the adaptive gradient optimizer to alter the descent rate towards the extremum of a cost function, often accelerating the descent rate in directions with small variations in the cost function and decelerating the descent rate in directions with large variations in the cost function. Popular optimizers with continuous-time dynamics are RMSprop \cite{ref:tieleman-2012} and AdaHessian \cite{ref:yao-2021}.} \add{In continuous time, an adaptive gradient method is an interconnected system: a parameter-update subsystem, driven through an adaptive gain, is coupled to an auxiliary subsystem that low-pass filters derivative information of the cost function. Both subsystems are ISS.}

\add{Interconnected systems are a mature subject \cite{ref:karafyllis-2011}. When the subsystems are ISS \cite{ref:ito-2008,ref:jiang-1996:iss-max,ref:sontag-1989}, integral ISS (iISS) \cite{ref:ito-2009,ref:ito-2010,ref:ito-2012,ref:ito-2013,ref:ito-2020}, or passive \cite{ref:jiang-1996:passivity,ref:ordonez-2013}, and in the mixed ISS/iISS case \cite{ref:ito-2006,ref:ito-2009,ref:chaillet-2014}, small-gain theorems \cite{ref:ito-2009,ref:jiang-2004:small-gain,ref:teel-1996} and integral- \cite{ref:ito-2006,ref:ito-2008,ref:ito-2009}, max- \cite{ref:ito-2012,ref:ito-2014,ref:jiang-1996:iss-max}, and sum-type \cite{ref:ito-2012,ref:ito-2013,ref:ito-2014} Lyapunov constructions are available. These tools apply to the systems studied here. In fact, the parameter subsystem is ISS with an arbitrarily small gain (Section~\ref{sec:problem-statement}), so the small-gain condition holds automatically and GAS of the interconnection follows from \cite{ref:teel-1996}. The general constructions require auxiliary gain, supply-rate, or scaling functions. Here they are specialized to adaptive gradient interconnections so that the resulting Lyapunov functions are written directly in terms of the cost function, the adaptation law, gains in the parameter updates, and growth rates associated with these functions.}

\noindent\textbf{Main Contribution:} \add{Under mild assumptions given in Section~\ref{sec:problem-statement}, we give three explicit Lyapunov function constructions for the interconnected systems arising from adaptive gradient methods: a max-type construction valid in general (Theorem~\ref{thm:adaptive-gradient-methods-are-gas}), an error-coordinate construction for scalar parameters (Proposition~\ref{prop:scalar-case-lyapunov-function}), and an integral-transform construction under a divergence condition on the gain (Proposition~\ref{prop:integral-transforms}). Each is written in terms of the cost function, its gradient, and a growth bound on the adaptation map, and each is related, where it appears, to the existing max-type and scaling constructions.}

Notation: Let $\mathbb{R}_{\geq0}=[0,\infty)$, $\vert \cdot \vert$ be the Euclidean norm, and $\mathcal{K}$, $\mathcal{K}_\infty$, and $\mathcal{KL}$ denote the usual classes of comparison functions \cite{ref:khalil-2002}. For a \add{continuous} function $v:\mathbb{R}^n\to\mathbb{R}$, the upper right-hand Dini derivative \cite[App~C.2]{ref:khalil-2002} along a \add{solution} $x(t)$ \add{of $\dot{x}=f(x)$} is
\begin{align}
    \label{eq:upper-right-dini-derivative}
    D^+ v\left(x\left(t\right)\right)
    \;:=\;
    \limsup_{h\to0^+}\frac{v\left(x\left(t+h\right)\right)-v\left(x\left(t\right)\right)}{h}.
\end{align}
Similarly, the lower left-hand Dini derivative is
\begin{align}
    \label{eq:lower-left-dini-derivative}
    D_- v\left(x\left(t\right)\right)
    \;:=\;
    \liminf_{h\to0^+}\frac{v\left(x\left(t\right)\right)-v\left(x\left(t-h\right)\right)}{h}.
\end{align}
These generalized (one-sided) derivatives allow the inclusion of nonsmooth Lyapunov functions in stability proofs \cite{ref:clarke-1990}. Note that for $\alpha\in\mathcal{K}$, $D_-\alpha(r)\geq 0$ since $\alpha$ must be a strictly increasing function. \add{When $v$ is continuously differentiable, $D^+v(x(t))=\dot{v}(x(t))$, where $\dot{v}(x):=\nabla v(x)\cdot f(x)$ is the derivative of $v$ along the solutions of $\dot{x}=f(x)$; we indicate the arguments on which $\dot{v}$ depends explicitly.}

\section{Motivating Example}
    \label{sec:motivating-example}
    
As a motivating example, consider the following system
\begin{subequations}
\label{eq:motivating-example:full-system} 
    \begin{align}
        \label{eq:motivating-example:parameter-update-ode}
        \dot{\vartheta} &= -\frac{2\vartheta}{1+\sqrt{\vert \varphi \vert}} \\
        \label{eq:motivating-example:gradient-squared-estimate}
        \dot{\varphi} &= 4\vartheta^2 - \varphi
    \end{align}
\end{subequations}
which is a representative system of those that arise in adaptive gradient optimizers. We wish to establish the stability properties, if any, of each individual subsystem and the overall interconnected \add{system}. Specifically, we want to know which, if any, are either GAS or ISS.
It is easy to determine the stability properties of subsystem \eqref{eq:motivating-example:gradient-squared-estimate} by considering the function $\varphi^2$. \add{Its derivative along the solutions of \eqref{eq:motivating-example:full-system} satisfies the bound}
\begin{align}
    \frac{d}{dt}\varphi^2 = 8 \vartheta^2 \varphi - 2\varphi^2 \leq -\varphi^2, \quad \forall \vert \varphi \vert \geq 8\vartheta^2.
\end{align}
This establishes that subsystem \eqref{eq:motivating-example:gradient-squared-estimate} is ISS with a quadratic gain function $\gamma_2(r) = 8r^2$ by \cite[Th~4.19]{ref:khalil-2002}. The subsystem~\eqref{eq:motivating-example:parameter-update-ode} is also ISS, which is shown by considering the function $\vartheta^2$ and \add{its derivative}
    \begin{equation}
        \frac{d}{dt} \vartheta^2 = \frac{-4 \vartheta^2}{1 + \sqrt{\vert\varphi\vert}}
        \leq \frac{-4 \vartheta^2}{1 + a^{-1/2}\vert\vartheta\vert}, \quad \forall \vert\vartheta\vert \geq \sqrt{a \vert\varphi\vert} 
    \end{equation}
    where $a>0$ is an arbitrary scaling constant. This implies gain $\gamma_1(r) = \sqrt{ar}$ whose growth can be made arbitrarily slow. Additionally, both subsystems satisfy the ISS dissipative-form inequalities
    \begin{align}
    \label{eq:motivating-example:dissipative-form:theta}
        \frac{d}{dt} \vartheta^2 &\leq \frac{-4 \vartheta^2}{1 + a^{-1/2}\vert\vartheta\vert} + \frac{4 a \vert \varphi \vert}{1 + \sqrt{\vert\varphi\vert}} \\
    \label{eq:motivating-example:dissipative-form:phi}
        \frac{d}{dt} \varphi^2 &\leq -\frac{3}{2}\varphi^2 + 32 \vartheta^4
    \end{align}
    where \eqref{eq:motivating-example:dissipative-form:theta} follows from \cite[Rem~2.4]{ref:sontag-1995} and \eqref{eq:motivating-example:dissipative-form:phi} is derived using Young's inequality.

Notably, the dissipative characteristics of these two subsystems differ significantly. The $\varphi$-subsystem exhibits exponential decay uniform in the initial condition $\varphi_0$. In contrast, the $\vartheta$-subsystem exhibits a state-dependent decay rate; the dissipation weakens as the input $\varphi$ grows, maintaining exponential convergence only locally or when $\varphi$ is bounded. Despite this, the compositions of the gains ${\gamma_1\circ\gamma_2(r) = \sqrt{8a}r}$ and ${\gamma_2\circ\gamma_1(r) = 8 a r}$ are both contractive maps with $a<1/8$ so the interconnected system is GAS by nonlinear small-gain theorems \cite[Th~1]{ref:teel-1996}. \add{Beyond the GAS conclusion, a Lyapunov function is wanted, since it supplies forward-invariant sublevel sets and is the starting point for robustness analysis and parameter selection. The Lyapunov formulation of the small-gain theorem \cite{ref:jiang-1996:iss-max} guarantees one of the form $V_0=\max\{\sigma(\vartheta^2),\varphi^2\}$ for any continuously differentiable $\sigma\in\mathcal{K}_\infty$ in a band determined by the two gains---here $64r^2<\sigma(r)<256r^2$, e.g., $\sigma(r)=128r^2$---with a general-purpose $\sigma$ supplied by the construction in the appendix of \cite{ref:jiang-1996:iss-max}. The following three Lyapunov functions are instead produced by constructions, given in Sections~\ref{sec:constructing-lyapunov-functions:maximum} and \ref{sec:constructing-lyapunov-functions} for the general class of Section~\ref{sec:problem-statement}, whose inputs are the cost function and growth bounds on its derivatives:}
\begin{subequations}
\label{eq:motivating-example:conclusion-of-lyapunov-function}
\begin{align}
    \label{eq:motivating-example:lyapunov-function:maximum}
    V_{1}(\vartheta, \varphi) &= \vartheta^2 + \max\left\lbrace 64\vartheta^4, \varphi^2\right\rbrace \\
    \label{eq:motivating-example:lyapunov-function:scalar}
    V_{2}(\vartheta, \varphi) &= 5\vartheta^2 + \ln\left(1 + (\varphi - 4\vartheta^2)^2\right) \\
    \label{eq:motivating-example:lyapunov-function:integrals-alternative}
    \add{V_{3}(\vartheta,\varphi)} & \add{=\begin{multlined}[t][5cm]
    \vartheta^2 + 2\vartheta^4 + \frac{2}{3}\vert\varphi\vert^{3/2} - \vert\varphi\vert + 2\vert\varphi\vert^{1/2} \\ -2\ln\left(\vert\varphi\vert^{1/2} + 1\right)
        \end{multlined}}
\end{align}
\end{subequations}
\add{Their derivatives} along the solutions of \eqref{eq:motivating-example:full-system} are bounded by
\begin{subequations}
\label{eq:motivating-example:conclusion-of-lyapunov-function-derivative}
\begin{align}
    \label{eq:motivating-example:lyapunov-function:maximum:derivative}
    \dot{V}_{1}(\vartheta,\varphi) &\leq \frac{-4\vartheta^2}{1 + \sqrt{\vert\varphi\vert}} - \begin{cases}
        \varphi^2 & \text{ if } \varphi^2 > 64 \vartheta^4 \\
        0 & \text{ otherwise}
    \end{cases} \\
    \label{eq:motivating-example:lyapunov-function:scalar:derivative}
    \dot{V}_{2}(\vartheta,\varphi) &\leq \frac{-4\vartheta^2}{1 + \sqrt{\vert\varphi\vert}} - \frac{2(\varphi - 4\vartheta^2)^2}{1 + \left(\varphi - 4\vartheta^2\right)^2} \\
    \label{eq:motivating-example:lyapunov-function:integrals:derivative}
    \add{\dot{V}_{3}(\vartheta,\varphi)} & \leq \add{\frac{-4\vartheta^2 - 8\vartheta^4 - \tfrac{1}{2}\vert\varphi\vert^2}{1 + \sqrt{\vert\varphi\vert}}}
\end{align}
\end{subequations}
\add{(with $\dot{V}_1$ understood as the upper right-hand Dini derivative on the surface $\varphi^2=64\vartheta^4$),} which are all negative definite functions. Thus, $V_{1}$, $V_{2}$, and $V_{3}$ are Lyapunov functions for proving that the interconnected system is GAS to the origin\add{, by \cite[Th~4.9]{ref:khalil-2002} for $V_2$ and $V_3$ and by the argument of Section~\ref{sec:constructing-lyapunov-functions:maximum} for the nonsmooth $V_1$. The bounds \eqref{eq:motivating-example:conclusion-of-lyapunov-function-derivative} are derived} in Section~\ref{sec:constructing-lyapunov-functions:maximum}, Section~\ref{sec:constructing-lyapunov-functions:scalar-maps}, and \add{Section~\ref{sec:constructing-lyapunov-functions:integral-transforms}}, respectively\add{, where $V_1$, $V_2$, and $V_3$ are obtained as instances of three constructions that apply to the general system of Section~\ref{sec:problem-statement}}. The level sets of these Lyapunov functions are depicted in Fig.~\ref{fig:level-sets}.

\begin{figure*}[t]
    \centering
    \includegraphics[width=\textwidth]{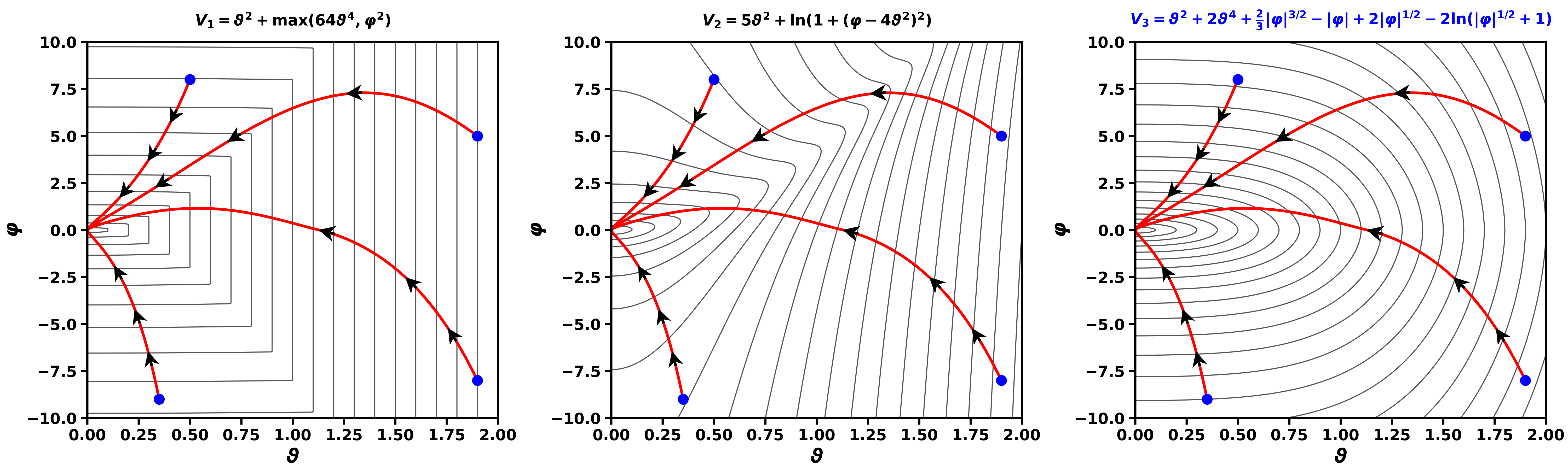}
    \caption{Level sets of the proposed Lyapunov functions $V_1, V_2, V_3$ with overlaid system trajectories. Solid \textcolor{red}{red lines} indicate trajectories initiating from various initial conditions (\textcolor{blue}{blue dots}). Due to symmetry, only the half-plane $\vartheta \in [0, 2]$ is depicted.}
    \label{fig:level-sets}
\end{figure*}

\section{Problem Statement}
    \label{sec:problem-statement}

\add{To specify the interconnected systems studied} in this work, we present an abstracted adaptive gradient-based system as
\begin{subequations}
\label{eq:problem-statement:full-system}
\begin{align}
    \label{eq:problem-statement:parameter-update}
    \dot{\vartheta} &= -K(\varphi) \nabla J(\vartheta)\\
    \label{eq:problem-statement:auxilliary-update}
    \dot{\varphi} &= \omega_{\ell}(q(\vartheta)-\varphi) 
\end{align}
\end{subequations}
Here, $J:\mathbb{R}^n\to\mathbb{R}$ is a cost function that depends on the state vector $\vartheta\in\mathbb{R}^n$, $\varphi\in\mathbb{R}^m$ is a low-pass filter\add{ed} estimate of some properties of $J$, $K(\cdot)$ denotes a \add{continuous} adaptive gain function whose output is \add{a positive definite and} symmetric matrix ($K(\varphi)\succ 0$) for all $\varphi\in\mathbb{R}^m$, and $\omega_\ell>0$ is the low-pass filter parameter. \add{The term $q(\vartheta)$ is an adaptation law that maps $\vartheta$ to the property of $J$ estimated by the low-pass filter $\varphi$, typically formed from elements of the gradient $\frac{\partial J}{\partial \vartheta_i}(\vartheta)$ or of the Hessian $\nabla^2 J(\vartheta)$; the gain $K(\varphi)$ uses $\varphi$ to shape the descent direction relative to $\nabla J$.} The cost function $J$ is assumed to satisfy the following assumptions. 

\begin{asmp}[Unique Minimum]
        \label{asmp:unique-minimum}
		For $J(\vartheta)$ where $J:\mathbb{R}^n\to\mathbb{R}$, $\exists\vartheta^{\ast}\in\mathbb{R}^n$ such that $\forall\ \vartheta\in\mathbb{R}^n$, $J(\vartheta) \geq J(\vartheta^{\ast})$ with equality if and only if $\vartheta = \vartheta^{\ast}$. Without loss of generality (wlog), let $\vartheta^{\ast} = 0$ and $J(\vartheta^{\ast}) = 0$.
\end{asmp}
\begin{asmp}[Single Critical Point]
        \label{asmp:single-critical-point}
		The cost function is continuously differentiable ($J\in\mathcal{C}^1$) and the gradient $\nabla J(\vartheta) = 0$ if and only if $\vartheta= 0$.
\end{asmp}
\begin{asmp}[Radially Unbounded $J$]
        \label{asmp:radially-unbounded}
        There exist $\alpha_{\vartheta, 1}$ and $\alpha_{\vartheta, 2}\in\mathcal{K}_{\infty}$ such that
        \begin{equation}
            \label{eq:asmp:radially-unbounded-bounds}
            \alpha_{\vartheta, 1}(\vert \vartheta \vert) \leq J(\vartheta) \leq \alpha_{\vartheta, 2}(\vert \vartheta \vert)
        \end{equation}
\end{asmp}
\begin{asmp}
    \label{asmp:q-continuous}
    The function $q:\mathbb{R}^n\to\mathbb{R}^m$ is a continuous function and, wlog, $q(0) = 0$.
\end{asmp}

Under Assumptions \ref{asmp:unique-minimum}--\ref{asmp:q-continuous}, we can analyze the stability properties of the two subsystems. \add{The right-hand side of \eqref{eq:problem-statement:full-system} is continuous, so solutions exist for every initial condition but need not be unique; all statements below hold for every solution.} Assumption~\ref{asmp:q-continuous} implies the existence of \add{a} $\rho_{q}\in\mathcal{K}_{\infty}$ such that
\begin{equation}
    \label{eq:property-growth-rate-q}
    \vert q(\vartheta) \vert \leq \frac{1}{2}\rho_{q}(\vert \vartheta \vert)
\end{equation}
Consider the ISS-Lyapunov function $V_{4}(\varphi) = \alpha_{\varphi}(\vert\varphi\vert) = \vert \varphi \vert^2$ for the auxiliary system. The derivative of ${V}_4$ satisfies
\begin{equation}
    \label{eq:auxiliary-system-is-iss:lyapunov-derivative}
    \dot{V}_{4}(\add{\vartheta,\varphi}) \leq -\omega_{\ell}\vert\varphi\vert^2 \quad \forall\ \vert\varphi\vert \geq \rho_{q}(\vert \vartheta \vert)
\end{equation}
so the auxiliary system is ISS \cite[Th~4.19]{ref:khalil-2002} with $\rho_{q}$ the system gain. For the parameter system, we examine the Lyapunov function $J$ whose \add{derivative} is
\begin{align}
\label{eq:problem-statement:j-dot-condition}
    \dot{J}(\vartheta, \varphi) &{}={} -\nabla J(\vartheta)^T\ K(\varphi)\ \nabla J(\vartheta) \notag \\ 
    &\leq -\lambda_{\min}\left(K(\varphi)\right)\ \vert \nabla J(\vartheta) \vert^2 =: - W(\vartheta,\varphi)
\end{align}
where $\lambda_{\min}(\cdot)$ gives the smallest eigenvalue of $K(\varphi)$. Here, $W(\cdot,\varphi)$ represents a positive definite function for every fixed $\varphi$. Consequently, $J$ is \add{nonincreasing along the solutions of \eqref{eq:problem-statement:full-system}, and strictly decreasing whenever $\vartheta(t)\neq 0$}. While the bound on $\dot{J}$ depends on both $\vartheta$ and $\varphi$, we define the function 
\begin{equation}
    \label{eq:w-function-def}
    \underline{W}(\vartheta) \coloneqq \min_{x\in\mathbb{R}^m \text{ and } \vert x \vert \leq \rho^{-1}(\vert \vartheta \vert)} W(\vartheta,x)
\end{equation}
with $\rho$ being \emph{any} $\mathcal{K}_{\infty}$ function to give the conditional bound
\begin{equation}
\label{eq:condition-of-dot-j-ndf}
    \dot{J}(\vartheta, \varphi) \leq - \underline{W}(\vartheta) \quad \forall \vert\vartheta\vert \geq \rho(\vert\varphi\vert).
\end{equation}
Thus, similar to the motivating example in Section~\ref{sec:motivating-example}, system \eqref{eq:problem-statement:parameter-update} is ISS with \emph{arbitrary} gain $\gamma\in\mathcal{K}_{\infty}$. This arbitrariness can be seen from \cite[Th~4.19]{ref:khalil-2002} by letting $\rho~=~\alpha_{\vartheta,2}^{-1}\circ\alpha_{\vartheta,1}\circ\gamma$ so that the system has gain $\gamma$. \add{Such an arbitrary gain is uncommon in the interconnection literature; an instance, not named as such, appears in \cite[Ex~2.46]{ref:mironchenko-2023}.} Letting $\gamma = \frac{1}{2}\rho_{q}^{-1}$ demonstrates that the interconnected system is GAS by \cite[Th~1]{ref:teel-1996} since $\frac{1}{2}\rho_{q}^{-1}\circ\rho_q (r) < r$ for $r > 0$. Additionally, $\underline{W}$ depends on the magnitude of $\nabla J$ and $\lambda_{\min}(K(\cdot))>0$. Since $\nabla J$ is only restricted by Assumption~\ref{asmp:single-critical-point}, $\lim_{r\to\infty} \min_{\vert \vartheta \vert = r} \underline{W}(\vartheta)$ can be $0$, with $\dot{J}$ potentially not having a dissipative inequality bound. 

\section{Baseline Result: A Lyapunov Function with a Maximum}
    \label{sec:constructing-lyapunov-functions:maximum}

The strict decrease of the cost function and the ISS property of the auxiliary system provide the insight for constructing the Lyapunov function which proves the following baseline theoretical result of this work. \add{GAS itself was already obtained from the small-gain theorem in Section~\ref{sec:problem-statement}; the content of the theorem is the Lyapunov function.}
\begin{thm}
    \label{thm:adaptive-gradient-methods-are-gas}
    Consider the system described by \eqref{eq:problem-statement:full-system}. If Assumptions \ref{asmp:unique-minimum}-\ref{asmp:q-continuous} hold, then the system \eqref{eq:problem-statement:full-system} is GAS at the origin\add{, and
    \begin{align}
    \label{eq:lyapunov-function:maximum}
    V_{5}(\vartheta, \varphi) &{}={} J(\vartheta) + \max\left\{\alpha_{\vartheta,\varphi}\circ J\left(\vartheta\right),\, \vert\varphi\vert^2\right\},
    \end{align}
    with $\alpha_{\vartheta,\varphi} := \alpha_{\varphi}\circ \rho_q \circ \alpha_{\vartheta,1}^{-1}\in\mathcal{K}_\infty$, is a Lyapunov function for it: $V_5$ is continuous, positive definite, and radially unbounded, and along every solution of \eqref{eq:problem-statement:full-system}
    \begin{equation}
        \label{eq:lyapunov-function:maximum:decrease}
        D^+V_5(\vartheta,\varphi) \leq -\lambda_{\min}\left(K(\varphi)\right)\vert\nabla J(\vartheta)\vert^2 = -W(\vartheta,\varphi).
    \end{equation}}
\end{thm}
\begin{proof}
The \add{proof} is accomplished by the construction of a Lyapunov function which relates the gain of \eqref{eq:auxiliary-system-is-iss:lyapunov-derivative} to the value of $J(\vartheta)$ rather than $\vert\vartheta\vert$. As $J$ is strictly decreasing along the solutions of the interconnected system, we naturally expect the argument in the ISS gain for \eqref{eq:problem-statement:auxilliary-update} to go to zero.

First, we relax the condition in \eqref{eq:auxiliary-system-is-iss:lyapunov-derivative} using the \add{lower bound in Assumption~\ref{asmp:radially-unbounded}}, \add{which gives $\rho_q(\vert\vartheta\vert)\leq \rho_q\circ\alpha_{\vartheta,1}^{-1}\circ J(\vartheta)$,} so that
\begin{equation}
    \label{eq:auxiliary-system:lyapunov-derivative:alternate}
    \dot{V}_{4}(\add{\vartheta,\varphi}) \leq -\add{\omega_{\ell}}\vert\varphi\vert^2 \quad \forall\ \vert \varphi \vert  \geq \rho_{q}\circ\alpha_{\vartheta,1}^{-1}\circ J(\add{\vartheta})\geq 0
\end{equation} 
\add{that is,}
\begin{equation}
    \label{eq:auxiliary-system:lyapunov-derivative:alternate-2}
    \dot{V}_{4}(\add{\vartheta,\varphi}) \leq -\add{\omega_{\ell}}\vert\varphi\vert^2\quad \forall\varphi\text{ s.t. } \vert\varphi\vert^2  \geq \alpha_{\vartheta, \varphi}\circ J (\vartheta)\geq 0.
\end{equation} 
The function $\alpha_{\vartheta,\varphi}\in\mathcal{K}_{\add{\infty}}$ since it is a composition of class $\mathcal{K}_{\add{\infty}}$ functions. This function is used to determine a lower bound, using the cost function $J(\vartheta)$, for which $\vert\varphi\vert^2$ decreases when $\vert\varphi\vert^2$ exceeds the lower bound. \add{Along every solution, $J(\vartheta(t))$ is nonincreasing by \eqref{eq:problem-statement:j-dot-condition} and $\alpha_{\vartheta,\varphi}$ is increasing, so $t\mapsto \alpha_{\vartheta,\varphi}(J(\vartheta(t)))$ is nonincreasing and, by the chain-rule inequality for Dini derivatives,}
\begin{equation}
    \label{eq:d-plus-with-parenthesize}
    D^{+}\bigl(\alpha_{\vartheta, \varphi}\circ J(\vartheta)\bigr) \leq -W(\vartheta,\varphi)\cdot (D_-\alpha_{\vartheta,\varphi})\bigl(J(\vartheta)\bigr) \leq 0,
\end{equation}
\add{regardless of whether $\alpha_{\vartheta,\varphi}$ is differentiable: $(D_-\alpha_{\vartheta,\varphi})(J(\vartheta))\in[0,+\infty]$ is nonnegative, and it will enter the estimates below only through subtraction, so an infinite value never invalidates a bound.}

The candidate Lyapunov function \eqref{eq:lyapunov-function:maximum} is meant to penalize three things: the excess ``value'' in the cost function $J(\vartheta)$; the magnitude of the lower bound of $\vert\varphi\vert^2$ given by ${\alpha_{\vartheta,\varphi}\circ J(\vartheta)}$; and the amount that $\vert\varphi\vert^2$ exceeds the bound.
\add{$V_5$ is continuous, and it} is radially unbounded as
\begin{equation}
    \label{eq:lyapunov-function:maximum:lower-bound}
    V_{5}(\vartheta, \varphi) \geq \alpha_{\vartheta,1}(\vert \vartheta \vert) + \vert \varphi \vert^2
\end{equation}
since $\max\lbrace a, b \rbrace \geq a$ and $\max\lbrace a, b \rbrace \geq b$ for any $a,b\in\mathbb{R}$. Additionally, we form the upper bound
\begin{equation}
    \label{eq:lyapunov-function:maximum:upper-bound}
    V_{5}(\vartheta, \varphi) \leq \alpha_{\vartheta,2}(\vert \vartheta \vert) + \alpha_{\vartheta,\varphi}\circ\alpha_{\vartheta,2}(\vert \vartheta \vert) + \vert \varphi \vert^2
\end{equation}
as $\max\lbrace a, b \rbrace \leq a + b$ if $a,b\geq 0$. \add{The radial unboundedness of $V_5$ together with the existence of an upper bound implies the existence of $\underline{\alpha}_{5},\overline{\alpha}_{5}\in\mathcal{K}_\infty$ such that
\begin{equation}
    \label{eq:lyapunov-function:maximum-minimum-bound}
    \underline{\alpha}_{5}(\vert(\vartheta,\varphi)\vert) \leq V_5(\vartheta,\varphi) \leq \overline{\alpha}_{5}(\vert(\vartheta,\varphi)\vert)
\end{equation}
by \cite[Lem~4.3]{ref:khalil-2002}.} Lastly, since the $\max$ function is nonsmooth at the switching surface
\begin{align}
    \label{eq:dini-switching-surface}
    \mathcal{S} = \left\lbrace (\vartheta, \varphi)\ \vert\ \vert\varphi\vert^2 = \alpha_{\vartheta,\varphi} \circ J\left(\vartheta \right)\right\rbrace
\end{align}
we do not expect the derivative of $V_{5}$ to be continuous and instead rely on the upper right Dini derivative of $V_{5}$ along solutions of \eqref{eq:problem-statement:full-system}. To this end, we consider three cases: when $(\vartheta, \varphi)$ is ``above'' the switching surface with $ \vert\varphi\vert^2 > \alpha_{\vartheta,\varphi}\circ J\left(\vartheta \right)$; when $(\vartheta, \varphi)$ is ``below'' the switching surface with $ \vert\varphi\vert^2 < \alpha_{\vartheta,\varphi}\circ J\left(\vartheta \right)$; and when $(\vartheta, \varphi)$ is on the switching surface.

\textbf{Case 1: Above the switching surface.} \add{This region is open, so a solution in it remains in it for a short time and there} $V_{5}(\vartheta, \varphi) = J(\vartheta) + \vert\varphi\vert^2$. Since $J(\vartheta)$ is continuously differentiable by Assumption \ref{asmp:single-critical-point}, \add{$D^+V_5$ coincides with the ordinary derivative along the solutions of \eqref{eq:problem-statement:full-system} and, by \eqref{eq:problem-statement:j-dot-condition} and \eqref{eq:auxiliary-system:lyapunov-derivative:alternate-2},}
\begin{align}
    \label{eq:lyapunov-function:maximum:above-surface:derivative}
    D^+V_{5}(\vartheta, \varphi) &= \dot{J}(\vartheta, \varphi) + \frac{d}{dt}\vert\varphi\vert^2 \leq - W(\vartheta, \varphi) - \omega_{\ell}\vert\varphi\vert^2.
\end{align}

\textbf{Case 2: Below the switching surface.} \add{This region is also open, and there, with $\vartheta\neq0$,} $V_{5}(\vartheta, \varphi) = J(\vartheta) + \alpha_{\vartheta,\varphi}\circ J\left(\vartheta\right)$. \add{By \eqref{eq:problem-statement:j-dot-condition} and \eqref{eq:d-plus-with-parenthesize},}
\begin{align}
    D^+V_{5}(\vartheta, \varphi) &{}={} \dot{J}(\vartheta, \varphi) + D^+\left(\alpha_{\vartheta,\varphi}\circ J\left(\vartheta\right)\right) \notag \\
    &{}\leq{} -W(\vartheta,\varphi)\cdot\left(1 + (D_-\alpha_{\vartheta,\varphi})\circ J(\vartheta)\right) \notag \\
    \label{eq:lyapunov-function:maximum:below-surface:derivative}
    &{}\leq{} -W(\vartheta, \varphi),
\end{align}
\add{which is negative since $W(\vartheta,\varphi)>0$ whenever $\vartheta\neq0$.}

\textbf{Case 3: On the switching surface.} \add{At a time $t$ at which $a(t)=b(t)$, $D^+\max\{a,b\}=\max\{D^+a,D^+b\}$, so}
\begin{multline}
    \label{eq:lyapunov-function:maximum:on-surface}
    D^+V_{5}(\vartheta, \varphi) = \dot{J}(\vartheta, \varphi) + \max\left\lbrace \frac{d}{dt}\vert\varphi\vert^2, \right. \\  D^+ \add{(\alpha_{\vartheta,\varphi}\circ J (\vartheta) )}\Big\rbrace
\end{multline}
\add{and, by \eqref{eq:auxiliary-system:lyapunov-derivative:alternate-2} and \eqref{eq:d-plus-with-parenthesize},}
\begin{align}
    D^+V_{5}(\vartheta, \varphi) &{}\leq{} \begin{multlined}[t][5cm] -W(\vartheta, \varphi) -\min\left\lbrace \omega_{\ell}\vert\varphi\vert^2, \right. \\ \left. W(\vartheta, \varphi)\cdot (D_{-}\alpha_{\vartheta,\varphi})\circ J(\vartheta)\right\rbrace \end{multlined} \notag \\
    \label{eq:lyapunov-function:maximum:on-surface:derivative}
    &\leq -W(\vartheta, \varphi).
\end{align}
\add{The right-hand side is zero only if $W(\vartheta,\varphi)=0$, i.e., only if $\vartheta=0$; on the switching surface this forces $\vert\varphi\vert^2=\alpha_{\vartheta,\varphi}(0)=0$, so the bound vanishes only at the origin.}

\textbf{Combining the cases:} \add{In all three cases, the upper right Dini derivative of $V_5$ along the solutions of \eqref{eq:problem-statement:full-system} satisfies $D^+V_5(\vartheta,\varphi)\leq -W_1(\vartheta,\varphi) \leq -W(\vartheta,\varphi)$, which is \eqref{eq:lyapunov-function:maximum:decrease}, where
\begin{equation}
    \label{eq:dini-proof:w-1-def}
    W_{1}(\vartheta,\varphi)\coloneqq W(\vartheta,\varphi) + \begin{cases}
        \omega_{\ell} \vert \varphi \vert^2 & \text{ if } \vert\varphi\vert^2 > \alpha_{\vartheta,\varphi}\circ J(\vartheta) \\
        0 & \text{ otherwise}
    \end{cases}
\end{equation}
While $W_1$ is a positive definite function, it is not continuous but rather lower semicontinuous. However, this is enough to prove that the system is GAS. Since $D^+V_5\leq 0$, $V_5$ must be nonincreasing along the solutions by the Comparison Lemma \cite[Lem~3.4]{ref:khalil-2002}, which is stated for the upper right Dini derivative of a continuous function and hence requires no differentiability of $V_5$. Hence, the system is both stable and bounded as
\begin{equation}
    \label{eq:dini-proof:stability-and-boundedness}
    \vert (\vartheta(t_0),\varphi(t_0)) \vert \leq \delta \implies \vert (\vartheta(t),\varphi(t)) \vert \leq \underline{\alpha}_5^{-1}\circ\overline{\alpha}_5 (\delta)
\end{equation}
for all $t\in[t_0,\infty)$. (Each maximal solution thus remains in a compact set and the right-hand side of \eqref{eq:problem-statement:full-system} is continuous, so every maximal solution is complete and the interval $[t_0,\infty)$ is justified.) Additionally, one can show via a proof by contradiction, following the steps in the proof of \cite[Th~8.2]{ref:verhulst-1996}, that the system is asymptotically stable since
\begin{multline}
    \label{eq:dini-proof:uniformly-asymptotically-stable}
    \vert (\vartheta(t_0),\varphi(t_0)) \vert \leq c_1 \implies \\ \vert (\vartheta(t),\varphi(t)) \vert \leq c_2 \ \forall t\in[t_0 + T(c_1,c_2),\infty)
\end{multline}
for all $c_1,c_2>0$, with
\begin{equation}
    \label{eq:dini-proof:settling-time}
    T(c_1,c_2) = \max\left\lbrace \frac{\overline{\alpha}_5(c_{\max}) - \underline{\alpha}_5 (c_{\min})}{\mu(c_{\min},c_{\max})} , 0\right\rbrace \,,
\end{equation}
\begin{equation}
    \label{eq:dini-proof:decay-rate}
    \mu(c_{\min},c_{\max}) = \min_{\begin{matrix}
            (\vartheta,\varphi)\in\mathbb{R}^{n+m} \\
            \text{s.t. } r_1(c_{\min}) \leq \vert(\vartheta,\varphi)\vert \\
            \phantom{s.t. } \vert(\vartheta,\varphi)\vert \leq r_2(c_{\max})
        \end{matrix}} W_1(\vartheta,\varphi) \,,
\end{equation}
$r_1(c) = \overline{\alpha}_5^{-1}\circ\underline{\alpha}_5(c)$, $r_2(c) = \underline{\alpha}_5^{-1}\circ\overline{\alpha}_5(c)$, $c_{\max} = \max\lbrace c_1,c_2\rbrace$, and $c_{\min} = \min\lbrace c_1, c_2\rbrace$. Note that the solution $\mu$ to the minimization problem of $W_1$ is well-defined since lower semicontinuous functions always attain their infimum over closed compact sets, and $\mu > 0$ as $W_1$ is a positive definite function and $r_1, r_2\in\mathcal{K}_{\infty}$. This allows us to conclude that $D^+ V_5 \leq -\mu$ within the annulus with inner radius $r_1$ and outer radius $r_2$, which contains the level set boundaries $V_5(\vartheta,\varphi) = \underline{\alpha}_5(c_{\min})$ and $V_5(\vartheta,\varphi) = \overline{\alpha}_5(c_{\max})$. Therefore, we conclude that the system is GAS by \cite[Def~4.4]{ref:khalil-2002}.}
\end{proof}
From the proof, we now see that $V_{1}$ in Section \ref{sec:motivating-example} is a specific instance of $V_{5}$ given the \add{$\mathcal{K}_\infty$} functions $\alpha_{\vartheta,1}(r)~=~\alpha_{\vartheta,2}(r)~=~r^2$ and $\rho_{q}(r) = 8r^2$. \add{$V_5$ is of the max type of \cite{ref:jiang-1996:iss-max,ref:ito-2012,ref:ito-2014}, with two specifics: the scaling of $J$ is explicit, $\sigma=\alpha_{\vartheta,\varphi}$, a choice the arbitrary gain of \eqref{eq:problem-statement:parameter-update} makes possible, and the additive term $J$ yields \eqref{eq:lyapunov-function:maximum:decrease} without any differentiability of $\alpha_{\vartheta,\varphi}$, whereas \cite{ref:jiang-1996:iss-max} requires a continuously differentiable $\sigma$ from a separate construction.}

\section{Additional Lyapunov Function Constructions}
    \label{sec:constructing-lyapunov-functions}

The proof of Theorem~\ref{thm:adaptive-gradient-methods-are-gas} \add{relies on Dini derivatives because $V_5$ is nonsmooth}. In Sections \ref{sec:constructing-lyapunov-functions:scalar-maps} and \ref{sec:constructing-lyapunov-functions:integral-transforms}, we give two additional constructions, \add{which avoid the maximum,} if there is some additional structure to the interconnected system.

\subsection{Lyapunov Function for \add{Scalar Parameters}}
    \label{sec:constructing-lyapunov-functions:scalar-maps}

\add{This section gives a different Lyapunov function for the case in which both the parameter $\vartheta$ and the auxiliary variable $\varphi$ are scalar, $n=m=1$.} The approach for the construction of this Lyapunov function is inspired by the Lyapunov function used in \cite{ref:mcnamee-2024:newton-automatica}. 

\begin{prop}
    \label{prop:scalar-case-lyapunov-function}
	\add{Consider the system \eqref{eq:problem-statement:full-system} with $n=m=1$, so that $J:\mathbb{R}\to\mathbb{R}$ and $q:\mathbb{R}\to\mathbb{R}$.} If Assumptions \ref{asmp:unique-minimum} and \ref{asmp:single-critical-point} are satisfied in addition to $q\in\mathcal{C}^{1}$, then the Lyapunov function 
    \begin{equation}
	    \label{eq:lyapunov-functions:scalar}
	    V_{6}(\vartheta,\tilde{\varphi}) = \vartheta^2 + \sgn(\vartheta)\int_{0}^{\vartheta}\left\vert q'(x) \right\vert dx + \ln(1 + \tilde{\varphi}^2),
	\end{equation}
	where $\sgn(\cdot)$ is the usual sign function and $\tilde{\varphi} = \varphi - q(\vartheta)$, has a negative definite \add{derivative along the solutions of \eqref{eq:problem-statement:full-system}} that satisfies the bound
	\begin{equation}
	    \label{eq:lyapunov-functions:scalar:lie-derivative-bound}
	    \dot{V}_{6}(\vartheta,\tilde{\varphi}) \leq - \frac{2\omega_{\ell}\tilde{\varphi}^2}{1+\tilde{\varphi}^2} -2K(\tilde{\varphi} + q(\vartheta))\vartheta J'(\vartheta).
	\end{equation}
\end{prop}
\begin{proof}
First, we introduce an error coordinate system of $\tilde{\varphi} = \varphi - q(\vartheta)$. The new dynamics are
\begin{subequations}
\label{eq:lyapunov-functions:scalar:system-error-ode}
\begin{align}
    \dot{\vartheta} &= -K(\tilde{\varphi} + q(\vartheta)) J'(\vartheta) \\
    \label{eq:lyapunov-functions:scalar:auxiliary-parameter-error-ode}
    \dot{\tilde{\varphi}} &= -\omega_{\ell}\tilde{\varphi} - q'(\vartheta)\cdot\dot{\vartheta}
\end{align}
\end{subequations}
\add{Note that the integral term of $V_6$ in \eqref{eq:lyapunov-functions:scalar} accounts for the interconnection term} in \eqref{eq:lyapunov-functions:scalar:auxiliary-parameter-error-ode}. Additionally, $V_{6}$ is a radially unbounded function, as the first and third terms in \eqref{eq:lyapunov-functions:scalar} are radially unbounded functions with respect to their arguments. \add{By Assumptions~\ref{asmp:unique-minimum} and~\ref{asmp:single-critical-point}, $\sgn(J'(\vartheta))=\sgn(\vartheta)$, so $\vert\vartheta(t)\vert$ is nonincreasing along \eqref{eq:lyapunov-functions:scalar:system-error-ode} and $\vartheta(t)$ never changes sign; the derivative of $V_6$, which is continuously differentiable for $\vartheta\neq 0$, is therefore evaluated for $\vartheta\neq0$.} The \add{derivative} of $V_{6}$ is
\begin{multline}
    \label{eq:lyapunov-functions:scalar:derivative}
        \dot{V}_{6} \textcolor{\addedcolor}{(\vartheta,\tilde{\varphi})} = - \frac{2\omega_{\ell}\tilde{\varphi}^2}{1+\tilde{\varphi}^2}-2K(\tilde{\varphi} + q(\vartheta))\, \vartheta \, J'(\vartheta)  \\ -K(\tilde{\varphi} + q(\vartheta))J'(\vartheta)\cdot\Big(  \sgn(\vartheta)\cdot \left\vert q'(\vartheta)\right\vert \\ \left. - \frac{2\tilde{\varphi}\cdot  q'(\vartheta)}{1 + \tilde{\varphi}^2}\right) 
\end{multline}
\add{Using $\sgn\left(J'(\vartheta)\right) = \sgn(\vartheta)$,} we bound \eqref{eq:lyapunov-functions:scalar:derivative} by
\begin{multline}
    \dot{V}_{6} \textcolor{\addedcolor}{(\vartheta,\tilde{\varphi})} \leq - \frac{2\omega_{\ell}\tilde{\varphi}^2}{1+\tilde{\varphi}^2} -2K(\tilde{\varphi} + q(\vartheta))\vartheta J'(\vartheta)  \\ -K(\tilde{\varphi} + q(\vartheta))\left\vert J'(\vartheta)\right\vert\cdot\left( \left\vert q'(\vartheta)\right\vert - \frac{2\vert\tilde{\varphi}\vert\cdot\left\vert q'(\vartheta)\right\vert}{1 + \tilde{\varphi}^2}\right) 
\end{multline}
By exploiting the fact that $\max_{\tilde{\varphi}\in\mathbb{R}}\left\vert \frac{2\tilde{\varphi}}{1 + \tilde{\varphi}^2} \right\vert = 1$, the bound on $\dot{V}_{6}$ is further simplified to what is shown in \eqref{eq:lyapunov-functions:scalar:lie-derivative-bound}.
\end{proof} 
\add{$V_2$ in Section~\ref{sec:motivating-example} is $V_6$ with $q(\vartheta)=4\vartheta^2$ and $K(\varphi)=(1+\sqrt{|\varphi|})^{-1}$. The term $\ln(1+\tilde{\varphi}^2)$ is a standard $\mathcal{K}_\infty$ scaling of the quadratic (cf.~\cite{ref:arcak-2002}), but $V_6$ is not obtained by scaling subsystem Lyapunov functions in the original coordinates as in \cite{ref:ito-2006,ref:ito-2010,ref:ito-2013}: it is written in the error coordinate $\tilde{\varphi}$, with $\sgn(\vartheta)\int_0^{\vartheta}\vert q'(x)\vert\,dx$ compensating the cross-coupling in \eqref{eq:lyapunov-functions:scalar:auxiliary-parameter-error-ode}. Two steps confine the construction to $n=1$: the identity $\sgn(J'(\vartheta))=\sgn(\vartheta)$ used after \eqref{eq:lyapunov-functions:scalar:derivative}, which has no counterpart for $n>1$ since $-K(\varphi)\nabla J(\vartheta)$ is not aligned with any direction determined by $\vartheta$ alone, and the compensation term, which for $n>1$ becomes a path integral of the Jacobian norm of $q$, in general the gradient of no function of $\vartheta$.}

\subsection{Lyapunov Function with Integral Transforms}
    \label{sec:constructing-lyapunov-functions:integral-transforms}
From Section \ref{sec:problem-statement}, we know that $J$ is a Lyapunov function for proving \eqref{eq:problem-statement:parameter-update} is strictly decreasing and that $\vert \varphi \vert^2$ worked as a Lyapunov function for proving \eqref{eq:problem-statement:auxilliary-update} is ISS to the origin. Generally, we do not expect a naive sum of $J(\vartheta)~+~\vert\varphi\vert^2$ to have a negative definite \add{derivative}. Inspired by \add{the sum-type integral constructions of \cite{ref:ito-2006,ref:ito-2009,ref:ito-2010,ref:ito-2012,ref:ito-2013,ref:ito-2020}}, we use integral transforms of these two terms in the next proposition. 

\begin{prop}
    \label{prop:integral-transforms}
    Consider the system described by \eqref{eq:problem-statement:full-system} \add{under} Assumptions \ref{asmp:unique-minimum}-\ref{asmp:q-continuous}. \add{Let $\gamma:\mathbb{R}_{\geq0}\to(0,\infty)$ be a continuous function such that}
    \begin{equation}
	    \label{eq:lyapunov-functions:integrals:gamma-2-gain-constraint}
	    \lambda_{\min}\left(K(\varphi)\right) \geq \gamma(\vert\varphi\vert^2) > 0,\quad \add{\forall \varphi\in\mathbb{R}^m,}
	\end{equation}
	\add{and let $b_g:(0,\infty)\to(0,\infty)$ be a continuous function such that}
    \begin{align}
	    \label{eq:lyapunov-functions:integrals:bg}
	    \ b_g (c) \add{\leq} \min_{\vartheta' \in \mathbb{R}^n \text{ s.t. } J(\vartheta') = c}\vert \nabla J(\vartheta') \vert^2 \add{,\quad \forall c>0,}
	\end{align}
	\add{such that the integrand of the first integral in \eqref{eq:lyapunov-functions:integrals:expanded} extends continuously to $r=0$.} Then the Lyapunov function
    \begin{multline}
        \label{eq:lyapunov-functions:integrals:expanded}
        V_{7}(\vartheta,\varphi) = J(\vartheta) + \int_0^{J(\vartheta)} \frac{\left(\rho_q \circ \alpha_{\vartheta,1}^{-1}(r)\right)^2}{4 b_{g}(r)}  dr \\ + \frac{1}{2\omega_{\ell}} \int_0^{|\varphi|^2} \gamma(r)  dr
    \end{multline}
    has a \add{derivative along the solutions of \eqref{eq:problem-statement:full-system}} bounded by 
    \begin{multline}
	    \label{eq:lyapunov-functions:integrals:lie-derivative-bound}
	    \dot{V}_{7}(\vartheta,\varphi) \leq -\left(\frac{1}{8}\left(\rho_q\circ\alpha_{\vartheta,1}^{-1}\circ J(\vartheta)\right)^2 \right. \\ \left .+ \vert\nabla J(\vartheta) \vert^2 + \frac{1}{2}\vert\varphi\vert^2\right)\cdot\gamma(\vert\varphi\vert^2)
	\end{multline}
    and is a negative definite function.
\end{prop}
\add{Before continuing onto the proof, we note that $\gamma$ satisfying \eqref{eq:lyapunov-functions:integrals:gamma-2-gain-constraint} always exists, e.g., $\gamma(s)=\min_{\vert\varphi\vert^2\leq s}\lambda_{\min}(K(\varphi))$; the minimum in \eqref{eq:lyapunov-functions:integrals:bg} is positive for each $c>0$ by Assumptions~\ref{asmp:single-critical-point}--\ref{asmp:radially-unbounded}, since the level set is compact.}

\begin{proof}
\add{By continuity of $\gamma$ and of the first integrand on $[0,\infty)$, $V_7$ is continuously differentiable, and its derivative along the solutions of \eqref{eq:problem-statement:full-system} is}
\begin{align}
\label{eq:integral-transform:dot-v-7-original}
    \dot{V}_{7}(\vartheta,\varphi) &{}={} -\nabla J(\vartheta)^T K(\varphi)\nabla J(\vartheta) \notag \\ &\quad - \frac{\left(\rho_q \circ \alpha_{\vartheta,1}^{-1}\circ J (\vartheta)\right)^2}{4 b_{g}\circ J (\vartheta)}\nabla J(\vartheta)^T K(\varphi)\nabla J(\vartheta) \notag \\ &\quad +\gamma\left(\vert\varphi\vert^2\right)\cdot \varphi^T\left( q(\vartheta) - \varphi\right)
\end{align}
Using \add{$b_g(J(\vartheta))\leq \vert\nabla J(\vartheta)\vert^2$ from \eqref{eq:lyapunov-functions:integrals:bg}} and the inequality 
\begin{equation}
	\nabla J(\vartheta)^T K(\varphi)\nabla J(\vartheta) \geq \lambda_{\min}(K(\varphi))\cdot \vert \nabla J(\vartheta) \vert^2,
\end{equation}
we bound $\dot{V}_7$ by
\begin{align}
    \dot{V}_{7}(\vartheta,\varphi) &{}\leq{} -\lambda_{\min}(K(\varphi))\cdot \vert \nabla J(\vartheta) \vert^2 \notag \\ &\quad - \add{\frac{1}{4}}\left(\rho_q \circ \alpha_{\vartheta,1}^{-1}\circ J (\vartheta)\right)^2\cdot \lambda_{\min}(K(\varphi))\notag \\ &\quad -\gamma (\vert\varphi\vert^2)\cdot\left[\vert\varphi\vert^2 - \varphi^T q\left(\vartheta\right)\right]
\end{align}
By using Young's inequality, we obtain
\begin{align}
	\label{eq:lyapunov-functions:integrals:derivative:youngs-inequality}
    \dot{V}_{7}(\vartheta,\varphi) &{}\leq{} -\lambda_{\min}(K(\varphi))\cdot \vert \nabla J(\vartheta) \vert^2 \notag \\ &\quad - \add{\frac{1}{4}}\left(\rho_q \circ \alpha_{\vartheta,1}^{-1}\circ J (\vartheta)\right)^2\cdot \lambda_{\min}(K(\varphi))\notag \\ &\quad + \frac{1}{2} \vert q(\vartheta)\vert^2\cdot\add{\gamma}(\vert\varphi\vert^2) - \frac{1}{2}\gamma (\vert\varphi\vert^2)\cdot\vert\varphi\vert^2
\end{align}
Since $\gamma$ satisfies \eqref{eq:lyapunov-functions:integrals:gamma-2-gain-constraint}, we further simplify the bound to
\begin{multline}
	\label{eq:lyapunov-functions:integrals:derivative:gamma-2-constraint}
    \dot{V}_{7}(\vartheta,\varphi) {}\leq{} -\left(\frac{1}{4} \left(\rho_q\circ\alpha_{\vartheta,1}^{-1}\circ J(\vartheta)\right)^2  \right. \\ \left. + \vert \nabla J(\vartheta) \vert^2 + \frac{1}{2} \vert\varphi\vert^2  - \frac{1}{2} \vert q(\vartheta)\vert^2  \right)\cdot\gamma (\vert\varphi\vert^2)
\end{multline}
\add{Inside the parenthesis of \eqref{eq:lyapunov-functions:integrals:derivative:gamma-2-constraint}, the first three terms are nonnegative; the only term that may prevent the parenthesis from being positive definite is $-\frac{1}{2}\vert q(\vartheta)\vert^2$.} However, by \add{\eqref{eq:property-growth-rate-q} and \eqref{eq:asmp:radially-unbounded-bounds}}, we bound $\vert q(\vartheta) \vert^2$ by
\begin{equation}
	\label{eq:lyapunov-functions:integrals:maximum-growth-bound}
	\left\vert q(\vartheta)\right\vert^2 \leq \frac{1}{4}\rho_q (\vert\vartheta\vert)^2 \leq \frac{1}{4}\left(\rho_q \circ \alpha_{\vartheta,1}^{-1} \circ J(\vartheta) \right)^2
\end{equation}
Application of \eqref{eq:lyapunov-functions:integrals:maximum-growth-bound} to \eqref{eq:lyapunov-functions:integrals:derivative:gamma-2-constraint} results in the bound given in \eqref{eq:lyapunov-functions:integrals:lie-derivative-bound}\add{, which is negative definite since $\gamma>0$}.
\end{proof}

Although we proved that $V_7$ always has a negative definite \add{derivative}, $V_7$ is only radially unbounded if 
\begin{equation}
	\label{eq:lyapunov-functions:integrals:gamma-2-integral-condition}
	\lim_{c\to\infty}\int_0^c \gamma (r) dr = \infty
\end{equation}
For the motivating example in Section~\ref{sec:motivating-example}, we have found such a $\gamma$, which was defined as $\gamma(r) = (1 + \sqrt[4]{r})^{-1}$ so that
\begin{equation}
    \gamma(\varphi^2) = K(\varphi) = \frac{1}{1 + \sqrt{\vert\varphi\vert}} > 0
\end{equation}
and whose integral satisfies \eqref{eq:lyapunov-functions:integrals:gamma-2-integral-condition}. \add{With $\alpha_{\vartheta,1}(r)=r^2$, $b_g(r)=4r$, $\rho_q(r)=8r^2$, and $\omega_\ell=1$, Proposition~\ref{prop:integral-transforms} yields exactly the Lyapunov function $V_{3}$ in \eqref{eq:motivating-example:lyapunov-function:integrals-alternative} and the bound \eqref{eq:motivating-example:lyapunov-function:integrals:derivative}.}

\add{$V_7$ has the sum-type integral form $\int_0^{V_a(\vartheta)}\lambda_1+\int_0^{V_b(\varphi)}\lambda_2$ of \cite{ref:ito-2006,ref:ito-2009,ref:ito-2010,ref:ito-2012,ref:ito-2013,ref:ito-2020}, with $V_a=J$, $V_b=\vert\varphi\vert^2$. The difference is where the scalings come from: in \cite{ref:ito-2006}, from the supply rates' comparison functions and auxiliary parameters constrained by a small-gain condition; here, $\lambda_1=1+(\rho_q\circ\alpha_{\vartheta,1}^{-1})^2/(4b_g)$ and $\lambda_2=\gamma/(2\omega_\ell)$, read off from the cost function and the bounds \eqref{eq:property-growth-rate-q} and \eqref{eq:lyapunov-functions:integrals:gamma-2-gain-constraint}---for the motivating example, $\lambda_1(s)=1+4s$ and $\lambda_2(s)=\tfrac{1}{2}(1+s^{1/4})^{-1}$, with no auxiliary parameters.}

\section{Additional Adaptive Gradient Systems Examples}
    \label{sec:examples}

Having shown three methodologies for generating Lyapunov functions, we now examine some example adaptive gradient systems and the Lyapunov functions generated for a scalar quadratic map $J(\vartheta) = \vartheta^2$ to see the complexity and applicability of the Lyapunov functions. Note that $\alpha_{\vartheta,1}(r)~=~r^2$ and $b_g(r)~=~4r$ for this section, while $\rho_q$ changes.

\subsection{RMSProp}
    \label{sec:examples:rmsprop}

RMSProp was introduced in \cite{ref:tieleman-2012} with the idea of scaling the learning rate for each parameter direction inversely by an approximate magnitude of the gradient based on the previous gradient information. The continuous-time dynamics, as given in \cite{ref:ma-2021}, is written element-wise as
\begin{subequations}
\begin{align}
    \label{eq:examples:rmsprop:parameter-update-ode}
    \dot{\vartheta}_i &= \frac{-k\frac{\partial J}{\partial \vartheta_i}(\vartheta)}{\varepsilon + \sqrt{\vert \varphi_i \vert}} \\
    \label{eq:examples:rmsprop:gradient-squared-estimate-ode}
    \dot{\varphi}_i &= \omega_{\ell}\left(\frac{\partial J}{\partial\vartheta_i}(\vartheta)\right)^2 - \omega_{\ell}\varphi_i
\end{align}
\end{subequations}
for $i=1,2,\ldots,n$ where $\varepsilon > 0$ is a small parameter to avoid a division by zero. For our abstraction, $K(\varphi)$ is a strictly diagonal positive definite matrix with $k (\sqrt{\vert \varphi_i \vert} + \varepsilon)^{\add{-1}}$ as its $i$-th diagonal element and $q(\vartheta)$ is the element-squared gradient. In truth, the system given in Section~\ref{sec:motivating-example} was RMSProp acting on $J(\vartheta) = \vartheta^2$ with $\varepsilon~=~1$, so we give a corollary of Theorem~\ref{thm:adaptive-gradient-methods-are-gas} before discussing other systems.
\begin{cor}
	The RMSProp system with a cost function $J$ satisfying Assumptions~\ref{asmp:unique-minimum}-\ref{asmp:radially-unbounded} is GAS at the origin.
\end{cor}

\subsection{AdaHessian}
    \label{sec:examples:adahessian}

Similar to RMSprop, AdaHessian \cite{ref:yao-2021} seeks to normalize the descent rate of the element-wise parameter dynamics but by the approximate magnitude of the diagonal elements of the Hessian rather than the approximate magnitude of the gradient. For simplicity, we consider the continuous-time dynamics of the AdaHessian system without the momentum term present in \cite{ref:yao-2021}. The dynamics are defined element-wise as
\begin{subequations}
\begin{align}
    \label{eq:adahessian:parameter-update-ode}
    \dot{\vartheta}_i &= -k \frac{\frac{\partial J}{\partial \vartheta_i}\left(\vartheta\right)}{\sqrt{\vert \varphi_i\vert}+\varepsilon} \\
    \label{eq:adahessian:estimate-update-ode}
    \dot{\varphi}_i &= \omega_{\ell}\left(\left( \frac{\partial^2 J}{\partial \vartheta_i^2}\left(\vartheta\right)\right)^2 -\varphi_i\right)
\end{align}
\end{subequations}
for $i=1,2,\ldots,n$ and $\varepsilon > 0$ is a small term to avoid \add{division} by zero. The gain $K(\varphi)$ is the same as RMSProp but now $q(\vartheta)$ is the element-squared diagonal of the Hessian matrix. This method has the advantage of extending Newton-based methods to nonconvex $J$ although, since the full Hessian is not estimated, the system does \add{not} have the assignable convergence rate of a standard Newton-based method \cite{ref:ghaffari-2012}. However, it is an example of an adaptive gradient method\add{; since $q(0)\neq0$ in general, Theorem~\ref{thm:adaptive-gradient-methods-are-gas} is applied in the shifted coordinates $\tilde{\varphi}=\varphi-q(0)$, with $\tilde{q}(\vartheta)=q(\vartheta)-q(0)$ and $\tilde{K}(\tilde{\varphi})=K(\tilde{\varphi}+q(0))$}.
\begin{cor}
	The AdaHessian system with a cost function $J\add{\in\mathcal{C}^2}$ that satisfies Assumptions~\ref{asmp:unique-minimum}-\ref{asmp:radially-unbounded} is GAS at $(\vartheta,\varphi)=(0,q(0))$.
\end{cor}

To illustrate our various Lyapunov function constructions, we now consider the AdaHessian system with $\varepsilon~=~1$ \add{and $k=1$} acting on a scalar quadratic map
\begin{subequations}
\begin{align}
    \label{eq:examples:adahessian:parameter-update-ode}
    \dot{\vartheta} &= \frac{-2\vartheta}{1 + \sqrt{\vert \varphi \vert}} \\
    \label{eq:examples:adahessian:estimate-update-ode}
    \dot{\varphi} &= \omega_{\ell}\left(4 - \varphi\right)
\end{align}
\end{subequations}
where the equilibrium is $(\vartheta,\varphi)=(0, 4)$. We introduce the error coordinate system of $\tilde{\varphi} = \varphi - 4$ which has the dynamics 
\begin{subequations}
\begin{align}
    \label{eq:examples:adahessian:error:parameter-update-ode}
    \dot{\vartheta} &= \frac{-2\vartheta}{1 + \sqrt{\vert \add{4+\tilde{\varphi}} \vert}} \\
    \label{eq:examples:adahessian:error:estimate-update-ode}
    \dot{\tilde{\varphi}} &= \add{-\omega_{\ell}\tilde{\varphi}}
\end{align}
\end{subequations}
It is trivial to show that the naive sum of squares Lyapunov function $\add{V_{8}}(\vartheta,\tilde{\varphi})~=~\vartheta^2~+~\tilde{\varphi}^2$ works for proving the system is GAS, but let us consider the Lyapunov functions discussed in Section~\ref{sec:constructing-lyapunov-functions}. For constructing these functions, any choice of $\rho_q~\in~\mathcal{K}_{\add{\infty}}$ would satisfy \eqref{eq:property-growth-rate-q} as $q(\vartheta)$ is a constant mapping, but for simplicity we choose $\rho_q(r)~=~r$. The Lyapunov functions of our constructive methods are
\begin{subequations}
\begin{align}
    V_{9}(\vartheta,\tilde{\varphi}) &= \vartheta^2 + \max\left\lbrace \vartheta^2, \tilde{\varphi}^2 \right\rbrace \\
    V_{10}(\vartheta,\tilde{\varphi}) &= \vartheta^2 + \ln\left(1 + \tilde{\varphi}^2\right) \\
    V_{11}(\vartheta, \tilde{\varphi}) &= \frac{17}{16}\vartheta^2 + \frac{1}{2\omega_{\ell}}\int_0^{\tilde{\varphi}^2} \frac{1}{1 + \sqrt{4 + \sqrt{r}}} dr
\end{align}
\end{subequations}
with \add{derivatives} of
\begin{subequations}
\begin{align}
    \dot{V}_{9}(\vartheta,\tilde{\varphi}) &\leq \frac{-4\vartheta^2}{1 + \sqrt{\vert 4 + \tilde{\varphi} \vert}} + \begin{cases}
        -\omega_{\ell} \tilde{\varphi}^2 & \text{ if } \vert \tilde{\varphi} \vert > \vert \vartheta \vert \\
        0 & \text{ otherwise}
    \end{cases}\\
    \dot{V}_{10}(\vartheta,\tilde{\varphi}) &= \frac{-4\vartheta^2}{1 + \sqrt{\vert 4 + \tilde{\varphi} \vert}} - \frac{2\omega_{\ell}\tilde{\varphi}^2}{1 + \tilde{\varphi}^2} \\
    \dot{V}_{11}(\vartheta, \tilde{\varphi}) &\leq \frac{-\frac{33}{8}\vartheta^2 - \frac{1}{2}\tilde{\varphi}^2}{1 + \sqrt{4 + \vert\tilde{\varphi}\vert}}
\end{align}
\end{subequations}
$V_{9}$ and $V_{10}$ are straightforward to reason about their construction and derivatives, and are very close to the naive \add{$V_{8}$}. They are more complicated than \add{$V_{8}$}, as that is the price paid for the generality of Sections~\ref{sec:constructing-lyapunov-functions:maximum} and \ref{sec:constructing-lyapunov-functions:scalar-maps}, but this price seems reasonable considering that \add{$V_{8}$} does not work for all $J$.

$V_{11}$ needs further explanation, specifically on how we select $\gamma$ such that \eqref{eq:lyapunov-functions:integrals:gamma-2-integral-condition} is satisfied so that $V_{11}$ is radially unbounded. We have the inequality $K(4 + \tilde{\varphi}) \geq \frac{1}{1 + \sqrt{4 + \vert \tilde{\varphi} \vert}}$ by the triangular inequality of $\vert\cdot\vert$. This leads to the selection of $\gamma(r) = (1 + \sqrt{4 + \sqrt{r}})^{-1}$ for constructing $V_{11}$. Additionally, we further expand the integral of $V_{11}$ to
\begin{multline}
    \int_0^{\tilde{\varphi}^2} \gamma(r) dr = \frac{40}{3} + \frac{4}{3}\left(\vert\tilde{\varphi}\vert - 5\right)\left(\sqrt{4 + \vert\tilde{\varphi}\vert}\right)  - 2\vert\tilde{\varphi}\vert \\ -12\ln(3) + 12\ln\left(1 + \sqrt{4 + \vert\tilde{\varphi}\vert}\right)
\end{multline}
Clearly, the antiderivative of $\gamma$ is not trivial for this example, although we found a closed-form expression in terms of elementary functions. However, for different cost functions, such as $J(\vartheta) = \sqrt{2}\vartheta^2$, it may not be possible to find such a closed-form expression. Even without the closed-form expression, we prove that the integral is radially unbounded for any constant mapping $q$ with $q(\cdot) > 0$, not just four as in this example, by a comparison to the function $r^{-1}$. Since
\begin{equation}
    1 + \sqrt{q(0) + \sqrt{\vert r \vert}} \in \mathcal{O}(r^{1/4})\subset \mathcal{O}(r)\quad \text{ as } r\to\infty
\end{equation} there exists a finite $r_0 > 0$ such that $1 + \sqrt{q(0) + \sqrt{\vert r \vert}} \leq r$, for all $r \geq r_0$, and consequently
\begin{equation}
    \int_{r_0}^{c} \left(1 + \sqrt{q(0) + \sqrt{\vert r \vert}}\right)^{-1} dr \geq \int_{r_0}^c r^{-1} dr = \ln(c/r_0)
\end{equation}
where $c\geq r_0$. Since the natural logarithm is unbounded from above, it follows that $\int_{0}^{\tilde{\varphi}^2} \left(1 + \sqrt{q(0) + \sqrt{\vert r \vert}}\right)^{-1} dr$ is radially unbounded. Hence, the Lyapunov functions generated in Section~\ref{sec:constructing-lyapunov-functions:integral-transforms} prove stability but we may not be able to express the Lyapunov function in terms of elementary functions.

\subsection{An Example on Limitation of Integral Transform Approach}
    \label{sec:examples:counterexample}

While Sections~\ref{sec:examples:rmsprop} and \ref{sec:examples:adahessian} covered adaptive gradient methods that exist in previous literature, we want to investigate the following system
\begin{subequations}
\begin{align}
    \dot{\vartheta} &= -ke^{-\varphi^2}J'(\vartheta) \\
    \dot{\varphi} &= \omega_{\ell}\left(J'(\vartheta)^2 - \varphi\right)
\end{align}
\end{subequations}
for scalar $J$'s as this system gives us particular insights for our Lyapunov function constructions. For the quadratic scalar map with $k~=~\omega_{\ell}~=~1$, the system is
\begin{subequations}
\label{eq:examples:counterexample:full-system}
\begin{align}
    \label{eq:examples:counterexample:parameter-update-ode}
    \dot{\vartheta} &= -2e^{-\varphi^2}\vartheta \\
    \label{eq:examples:counterexample:auxiliary-update-ode}
    \dot{\varphi} &= 4\vartheta^2 - \varphi
\end{align}
\end{subequations}
with the origin as the equilibrium. The Lyapunov constructions from Sections~\ref{sec:constructing-lyapunov-functions:maximum} and \ref{sec:constructing-lyapunov-functions:scalar-maps} result in the previous Lyapunov functions of $V_{1}$ and $V_{2}$, respectively. \add{Their derivatives along the solutions of \eqref{eq:examples:counterexample:full-system} satisfy}
\begin{subequations}
\begin{align}
    \dot{V}_{1}(\vartheta,\varphi) &\leq -4 e^{-\varphi^2}\vartheta^2 - \begin{cases}
        \varphi^2 & \text{ if } \varphi^2 > 64\add{\vartheta^4} \\
        0 & \text{ otherwise}
    \end{cases} \\
    \dot{V}_{2}(\vartheta,\varphi) &\leq -4 e^{-\varphi^2}\vartheta^2 - \frac{2\left(4\vartheta^2 - \varphi\right)^2}{1 + \left(4\vartheta^2 - \varphi\right)^2}
\end{align}
\end{subequations}
which are negative definite functions.

Unlike our other Lyapunov function constructions, the method presented in Section~\ref{sec:constructing-lyapunov-functions:integral-transforms} fails to be able to prove that the system \eqref{eq:examples:counterexample:full-system} is GAS. To understand this, we note that \add{every $\gamma$ satisfying \eqref{eq:lyapunov-functions:integrals:gamma-2-gain-constraint} is bounded by $\gamma(r)\leq e^{-r}$ and}
\begin{equation}
    \lim_{c\to\infty}\int_0^c e^{-r}dr = 1 - \lim_{c\to\infty} e^{-c} = 1
\end{equation}
Thus, we cannot use Section~\ref{sec:constructing-lyapunov-functions:integral-transforms} to construct a radially unbounded Lyapunov function for this system\add{: condition \eqref{eq:lyapunov-functions:integrals:gamma-2-integral-condition} is a genuine restriction of Proposition~\ref{prop:integral-transforms}, which Theorem~\ref{thm:adaptive-gradient-methods-are-gas} and Proposition~\ref{prop:scalar-case-lyapunov-function} do not share.}

\subsection{\texorpdfstring{\textcolor{\addedcolor}{An Example Where Only the Scalar Construction Is Elementary}}{An Example Where Only the Scalar Construction Is Elementary}}
    \label{sec:examples:exponential-cost}
\begingroup\color{\addedcolor}
A complementary consideration are the cases for which the comparison functions of the general theory cease to be elementary while the construction of Section~\ref{sec:constructing-lyapunov-functions:scalar-maps} still yields a closed form. Consider RMSProp with $n=1$ acting on
\begin{equation}
    \label{eq:examples:exponential-cost}
    J(\vartheta) = e^{\vartheta} - \vartheta - 1,
\end{equation}
which satisfies Assumptions~\ref{asmp:unique-minimum}--\ref{asmp:q-continuous} with $\alpha_{\vartheta,1}(r) = e^{-r}+r-1$, $\alpha_{\vartheta,2}(r) = e^{r}-r-1$, $q(\vartheta) = (e^{\vartheta}-1)^2$, and $\rho_q(r) = 2(e^{r}-1)^2$. Here
\begin{equation}
    \label{eq:examples:exponential-cost:inverse}
    \alpha_{\vartheta,1}^{-1}(s) = s + 1 + W_0\!\left(-e^{-(s+1)}\right),
\end{equation}
with $W_0$ the principal branch of the Lambert $W$ function, so every object of the general analysis that involves $\alpha_{\vartheta,1}^{-1}$ is non-elementary: the ISS gain $\gamma = \alpha_{\vartheta,1}^{-1}\circ\alpha_{\vartheta,2}\circ\rho$ of Section~\ref{sec:problem-statement}, the band of admissible scalings $\sigma$ in the max-form function of \cite{ref:jiang-1996:iss-max}, the function $\alpha_{\vartheta,\varphi}$ in Theorem~\ref{thm:adaptive-gradient-methods-are-gas}, and $b_g$ in Proposition~\ref{prop:integral-transforms}. The construction of Proposition~\ref{prop:scalar-case-lyapunov-function} bypasses all of these objects: since $\sgn(q'(\vartheta))=\sgn(\vartheta)$ here, $\sgn(\vartheta)\int_0^{\vartheta}\vert q'(x)\vert\, dx = q(\vartheta)$, and \eqref{eq:lyapunov-functions:scalar} evaluates to the elementary Lyapunov function
\begin{equation}
    \label{eq:examples:exponential-cost:lyapunov}
    V_{12}(\vartheta,\tilde{\varphi}) = \vartheta^2 + \left(e^{\vartheta}-1\right)^2 + \ln\left(1 + \tilde{\varphi}^2\right)
\end{equation}
with $\tilde{\varphi} = \varphi - (e^{\vartheta}-1)^2$, whose derivative is bounded, by \eqref{eq:lyapunov-functions:scalar:lie-derivative-bound}, as
\begin{equation}
    \label{eq:examples:exponential-cost:derivative}
    \dot{V}_{12}(\vartheta,\tilde{\varphi}) \leq -\frac{2\omega_{\ell}\tilde{\varphi}^2}{1+\tilde{\varphi}^2} - \frac{2k\,\vartheta\left(e^{\vartheta}-1\right)}{\varepsilon + \sqrt{\vert\tilde{\varphi}\vert}},
\end{equation}
a negative definite function of $(\vartheta,\tilde{\varphi})$. GAS is certified by Theorem~\ref{thm:adaptive-gradient-methods-are-gas} in any case; Proposition~\ref{prop:scalar-case-lyapunov-function} adds a Lyapunov function that can be written down.
\endgroup

\section{Conclusion}

This work studies the stability of interconnected systems arising in the context of adaptive gradient methods in optimization. \add{GAS of these systems follows from the small-gain theorem, and we give three methodologies for constructing explicit Lyapunov functions that certify it.} The constructions for the general case, the scalar case, and a multivariable case satisfying certain function inequalities are given in Section~\ref{sec:constructing-lyapunov-functions:maximum}, \ref{sec:constructing-lyapunov-functions:scalar-maps}, and \ref{sec:constructing-lyapunov-functions:integral-transforms}. \add{Each construction is expressed in the cost function and growth bounds on its derivatives.}

\section*{References}
\bibliographystyle{IEEEtranS}
\bibliography{references}

\end{document}